\documentclass{article}

\usepackage{amsmath}
\usepackage{amssymb}
\usepackage{graphicx}
\usepackage{bm}
\usepackage{mathdots}
\usepackage{booktabs}
\usepackage{rotating}
\usepackage[ruled,linesnumbered]{algorithm2e}
\usepackage[a4paper,left=2.8cm,right=2.8cm,top=2.5cm,bottom=2.5cm]{geometry}
\usepackage{fancyhdr}
\usepackage[framemethod=tikz]{mdframed}
\newtheorem{theorem}{Theorem}[section]
\newtheorem{corollary}{Corollary}[section]

\newtheorem{lemma}{Lemma}[section]

\newtheorem{assumption}{Assumption}[section]

\makeatletter 
\@addtoreset{equation}{section}
\makeatother  

\newtheorem{remark}{Remark}[section]

\newenvironment{proof}{{\noindent\it Proof.}\quad}{\hfill $\square$\\}

\usepackage{url}
\usepackage{mathtools}

\newcommand{\U}{\mathcal{U}}
\newcommand{\Lt}{L^2}
\newcommand{\Hs}{H^s}
\newcommand{\Sd}{\mathbb{S}^d}

\newcommand{\Pn}{\mathbb{P}_n}

\begin{document}
\title{Bypassing the quadrature exactness assumption of hyperinterpolation on the sphere}

\author{Congpei An\footnotemark[1]
       \quad\quad Hao-Ning Wu\footnotemark[2]}

\renewcommand{\thefootnote}{\fnsymbol{footnote}}
\footnotetext[1]{School of Mathematics, Southwestern University of Finance and Economics, Chengdu, China (ancp@swufe.edu.cn).}
\footnotetext[2]{Department of Mathematics, The University of Hong Kong, Hong Kong, China (hnwu@connect.hku.hk).}


\maketitle

\begin{abstract}
This paper focuses on the approximation of continuous functions on the unit sphere by spherical polynomials of degree $n$ via hyperinterpolation. Hyperinterpolation of degree $n$ is a discrete approximation of the $L^2$-orthogonal projection of degree $n$ with its Fourier coefficients evaluated by a positive-weight quadrature rule that exactly integrates all spherical polynomials of degree at most $2n$.
This paper aims to bypass this quadrature exactness assumption by replacing it with the Marcinkiewicz--Zygmund property proposed in a previous paper. Consequently, hyperinterpolation can be constructed by a positive-weight quadrature rule (not necessarily with quadrature exactness). This scheme is referred to as \emph{unfettered hyperinterpolation}. This paper provides a reasonable error estimate for unfettered hyperinterpolation. The error estimate generally consists of
two terms: a term representing the error estimate of the original hyperinterpolation
of full quadrature exactness and another introduced as compensation for the loss of exactness
degrees. A guide to controlling the newly introduced term in practice is provided. In particular, if the quadrature points form a quasi-Monte Carlo (QMC) design, then there is a refined error estimate. Numerical experiments verify the error estimates and the practical guide.

\end{abstract}



\section{Introduction}

Let $\Sd:=\{x\in\mathbb{R}^{d+1}:\|x\|_2=1\}$ be the unit sphere in the Euclidean space $\mathbb{R}^{d+1}$ for $d\geq 2$, endowed with the surface measure $\omega_d$; that is, $\lvert\mathbb{S}^d\rvert:=\int_{\Sd}\text{d}\omega_d$ denotes the surface area of the unit sphere $\Sd$. Many real-world applications can be modeled as spherical problems. A critical task of spherical modeling is to find an effective data fitting strategy to approximate the underlying mapping between input and output data. Hyperinterpolation, introduced by Sloan in \cite{sloan1995polynomial}, is a simple yet powerful method for fitting spherical data, and it has received a great deal of interest since its birth, see, e.g., \cite{an2021lasso,MR2274179,le2001uniform,MR4226998,MR1761902,reimer2002generalized,zbMATH01421286,sloan2012filtered,MR1845243}. Given sampled data $\{(x_j,y_j)\}_{j=1}^m\subset \Sd\times \mathbb{R}$, the underlying mapping can be modeled as a spherical hyperinterpolant of degree $n$ in the form of 
\begin{equation}\label{equ:map}
x\in\Sd\mapsto \sum_{j=1}^mw_jy_jG_n(x,x_j)\in \mathbb{R},
\end{equation}
where $w_j>0$, $j=1,2,\ldots,m$, are some prescribed weights, 
$$G_n(x,y) = \sum_{\ell=0}^n\sum_{k=1}^{Z(d,\ell)}Y_{\ell,k}(x)Y_{\ell,k}(y)$$ 
is a kernel generated by the spherical harmonics $\{Y_{\ell,k}\}$ of degree $\ell$ at most $n$, and the the precise number $Z(d,\ell)$ of spherical harmonics of exact degree $\ell$ is given in \eqref{equ:numberZ}. 

The simplicity of spherical hyperinterpolation is manifested in the modeled mapping \eqref{equ:map}. Unlike many other fitting techniques that usually need to solve a system of linear equations to obtain the modeled mapping, e.g., the least squares, the spherical hyperinterpolation \eqref{equ:map} can be  directly written down and immediately generates the output from any input $x\in\Sd$ without any mathematical manipulations but only addition and multiplication. Moreover, adding a new data pair or withdrawing an existing one can be directly achieved without a new computation from scratch.

However, the construction of hyperinterpolation of degree $n$ requires a positive-weight quadrature rule
\begin{equation}\label{equ:quad}
\sum_{j=1}^mw_jf(x_j)\approx \int_{\Sd}f\text{d}\omega_d
\end{equation}
to be exact for polynomials up to degree $2n$, that is,
\begin{equation}\label{equ:quadexactness}
\sum_{j=1}^mw_jf(x_j)= \int_{\Sd}f\text{d}\omega_d\quad \forall f\in\mathbb{P}_{2n}(\Sd),
\end{equation}
where $\Pn(\Sd)$ be the space of spherical polynomials of degree at most $n$. A convenient $\Lt$-orthonormal basis (with respect to $\omega_d$) for $\Pn$ is provided by the spherical harmonics $\{Y_{\ell,k}:k=1,2,\ldots Z(d,\ell);\ell=0,1,2,\ldots,n\}$. The hyperinterpolation operator $\mathcal{L}_n:\mathcal{C}(\Sd)\rightarrow\Pn(\Sd)$ maps a continuous function $f\in \mathcal{C}(\Sd)$ to 
\begin{equation}\label{equ:hyper}
\mathcal{L}_nf := \sum_{\ell=0}^n\sum_{k=1}^{Z(d,\ell)}\left\langle f,Y_{\ell,k}\right\rangle_mY_{\ell,k}\in\Pn(\Sd),
\end{equation}
where $\left\langle f,g\right\rangle_m:=\sum_{j=1}^mw_jf(x_j)g(x_j)$ is the numerical evaluation of the inner product $\left\langle f,g\right\rangle:=\int_{\Sd}f(x)g(x)\text{d}\omega_d$ by the quadrature rule \eqref{equ:quad} with the exactness assumption \eqref{equ:quadexactness}. 
In other words, the hyperinterpolation \eqref{equ:hyper} of $f\in \mathcal{C}(\Sd)$ can be regarded as a discrete version of the famous $\Lt$-orthogonal projection 
\begin{equation}\label{equ:proj}
\mathcal{P}_nf := \sum_{\ell=0}^n\sum_{k=1}^{Z(d,\ell)}\left\langle f,Y_{\ell,k}\right\rangle Y_{\ell,k}\in\Pn(\Sd)
\end{equation}
of $f$ from $\mathcal{C}(\Sd)$ onto $\Pn(\Sd)$. Sometimes we may consider equal-weight quadrature rules of the form
\begin{equation}\label{equ:equalweightquad}
\frac{1}{m}\sum_{j=1}^mf(x_j)\approx \int_{\Sd}f\text{d}\omega_d.
\end{equation}

Regarding this very restrictive nature of \eqref{equ:quadexactness} that it is impractical and sometimes impossible to obtain data on the desired quadrature points in practice, our aim in this paper is to bypass this quadrature exactness assumption by replacing it with the \emph{Marcinkiewicz--Zygmund property} (see \cite{an2022quadrature}):
\begin{assumption}
We assume that there exists an $\eta\in[0,1)$ such that
\begin{equation}\label{equ:etaassumption}
\left\lvert\sum_{j=1}^mw_j\chi(x_j)^2-\int_{\Sd}\chi^2\text{d}\omega_d\right\rvert\leq \eta \int_{\Sd}\chi^2\text{d}\omega_d\quad \forall \chi\in\mathbb{P}_{n}(\Sd).
\end{equation}
If $n'=n$, i.e., the quadrature exactness is not relaxed, then the exactness \eqref{equ:quadexactness} implies $\eta =0$. 
\end{assumption}
Then the construction of hyperinterpolation is feasible with many more quadrature rules outside the traditional candidates. Traditionally, quadrature rules using spherical $t$-designs are used to construct hyperinterpolation. As we can see in this paper, quadrature rules using scattered points, equal area points, minimal energy points, maximal determinant points, and many other kinds of points are also feasible for constructing hyperinterpolation. The Marcinkiewicz--Zygmund property \eqref{equ:etaassumption} is equivalent to 
\begin{equation*}
(1-\eta) \int_{\Sd}\chi^2\text{d}\omega_d\leq \sum_{j=1}^mw_j\chi(x_j)^2\leq (1+\eta) \int_{\Sd}\chi^2\text{d}\omega_d\quad \forall \chi\in\mathbb{P}_{n}(\Sd),
\end{equation*}
which can be regarded as the Marcinkiewicz--Zygmund inequality \cite{filbir2011marcinkiewicz,Marcinkiewicz1937,mhaskar2001spherical} applied to polynomials $\chi^2$ of degree at most $2n$ with $\chi\in\Pn(\Sd)$, and it has been utilized in our recent work \cite{an2022quadrature} that quadrature rules are assumed to have exactness degree $n+n'$ with $0<n'\leq n$ for the construction of hyperinterpolation.

To tell the difference between the original hyperinterpolation $\mathcal{L}_n$ and the hyperinterpolation relying only on the Marcinkiewicz--Zygmund property \eqref{equ:etaassumption}, we refer to the latter as the \emph{unfettered hyperinterpolation}, indicating that the application of hyperinterpolation is no longer limited by the quadrature exactness assumption, and denote it by \begin{equation}\label{equ:unfetteredhyper}
\U_nf:=\sum_{\ell=0}^n\sum_{k=1}^{Z(d,\ell)}\langle f,Y_{\ell,k}\rangle_mY_{\ell,k}\in\Pn(\Sd),
\end{equation}
where the quadrature rule \eqref{equ:quad} for evaluating $\langle f,Y_{\ell,k}\rangle_m$ is only assumed to satisfy the property \eqref{equ:etaassumption}.

We derive in this paper that 
\begin{equation}\label{equ:error2}
\|\U_nf-f\|_{\Lt}\leq\left(\sqrt{1+\eta}\left(\sum_{j=1}^mw_j\right)^{1/2}+\lvert\Sd\rvert^{1/2}\right)E_n(f)+\sqrt{\eta^2+4\eta}\|\chi^*\|_{\Lt},
\end{equation}
where $E_n(f)$ denotes the best uniform error of $f\in \mathcal{C}(\Sd)$ by a polynomial in $\mathbb{P}_n(\Sd)$, that is, $E_n(f):=\inf_{\chi\in \mathbb{P}_n(\Sd)}\|f-\chi\|_{\infty}$, and $\chi^*\in\Pn(\Sd)$ is the best approximation polynomial of $f$ in $\Pn(\Sd)$ in the sense of $\|f-\chi^*\|_{\infty}=E_n(f)$. Thus, no matter what kind of point distributions is adopted, it is sufficient for a reasonable approximation error bound to control the numerical integration error so that the constant $\eta$ in the Marcinkiewicz--Zygmund property \eqref{equ:etaassumption} is reasonably small.

The $\Lt$ error estimate \eqref{equ:error2} reduces to the classical result $\|\mathcal{L}_nf-f\|_{\Lt}\leq 2\lvert\Sd\rvert^{1/2}E_n(f)$ of hyperinterpolation derived in \cite{sloan1995polynomial} when the quadrature exactness degree is assumed to be $2n$, because such an assumption leads to $\eta=0$ and $\sum_{j=1}^mw_j = \int_{\Sd}\text{d} \omega_d=\lvert\Sd\rvert$. If the quadrature exactness degree is assumed to be $n+n'$ with $0<n'\leq n$, then the estimate \eqref{equ:error2} can be refined as 
\begin{equation*}
\|\U_nf-f\|_{\Lt} \leq \left(\sqrt{1+\eta}+1\right)\lvert\Sd\rvert^{1/2}E_{n'}(f),
\end{equation*}
and this convergence rate in terms of $E_{n'}(f)$ coincides with the result in our recent work \cite{an2022quadrature} that 
\begin{equation}\label{equ:BIT}
\|\mathcal{L}_nf-f\|_{\Lt}\leq \left(\frac{1}{\sqrt{1-\eta}}+1\right)\lvert\Sd\rvert^{1/2}E_{n'}(f)
\end{equation}
under the same assumption. A Sobolev analog to the error estimate \eqref{equ:error2}, i.e., the error measured by a Sobolev norm, is also established in this paper.

We also highlight the connection between the unfettered hyperinterpolation and QMC designs. Historically, quadrature exactness is often a starting point in designing quadrature rules. Nevertheless, this trend has recently received growing concerns regarding whether exactness is a reliable designing principle, see, e.g., \cite{trefethen2022exactness}. The concept of QMC designs, introduced by Brauchart, Saff, Sloan, and Womersley in \cite{MR3246811}, is an important quadrature-designing principle against this historical trend. QMC designs include many points distributions that are easy to obtain numerically, and quadrature rules using QMC designs provide the same asymptotic order of convergence as rules with quadrature exactness when the integrand belongs to the Sobolev space $H^s(\Sd)$ with $s>d/2$. Moreover, quadrature exactness is not a necessary assumption for QMC designs. If the quadrature points form a QMC design, then we show quadrature rules using them also satisfy the Marcinkiewicz--Zygmund property \eqref{equ:etaassumption}. Hence hyperinterpolation using QMC designs is a special case in the general framework of unfettered hyperinterpolation. However, the general error estimate \eqref{equ:error2} may not be sharp for hyperinterpolation using QMC designs, and we can refine them. Regarding the particularity of QMC designs, we may refer to the hyperinterpolation of $f\in\Hs(\Sd)$ using QMC designs, though a special case of unfettered hyperinterpolation, as the \emph{QMC hyperinterpolation}, and denote it by 
\begin{equation}\label{equ:QMChyper}
\mathcal{Q}_nf:=\sum_{\ell=0}^n\sum_{k=1}^{Z(d,\ell)}\left\langle f,Y_{\ell,k}\right\rangle_mY_{\ell,k}\in\Pn(\Sd),
\end{equation}
where the quadrature rule \eqref{equ:quad} for evaluating $\langle f,Y_{\ell,k}\rangle_m$ adopt a QMC design for $\Hs(\Sd)$ as the set of quadrature points. We show in this paper that for $f\in \Hs(\Sd)$,
\begin{equation*}
\|\mathcal{Q}_nf-f\|_{\Lt}\leq c''(s,d)\left(n^{-s}+\frac{1}{m^{s/d}}\sqrt{\frac{Z(d+1,n)}{a_n^{(s)}}}\right)\|f\|_{\Hs},
\end{equation*}
where $c''(s,d)>0$ is some constant depending only on $c$ and $s$, and $a_n^{(s)}$ is of order $(1+n)^{-2s}$.

\textbf{Organization.} The paper is organized as follows. Section \ref{sec:background} collects some technical facts regarding spherical
harmonics, our Sobolev space setting, spherical $t$-designs, and QMC designs. Section \ref{sec:unfetteredtheory} gives the approximation theory of the unfettered hyperinterpolation under the only assumption of the Marcinkiewicz--Zygmund property \eqref{equ:etaassumption}. Section \ref{sec:QMCtheory} develops the approximation theory of the QMC hyperinterpolation under the only assumption that $\{x_j\}_{j=1}^m$ is a QMC design. Section \ref{sec:numerical} contains numerical experiments that validate our theory.

\section{Background}\label{sec:background}

We are concerned with real-valued functions on the sphere $\Sd$ in the Euclidean space $\mathbb{R}^{d+1}$ for $d\geq 2$.

\subsection{Spherical harmonics and hyperinterpolation} 

Let $\Lt(\Sd)$ denote the Hilbert space of all square-integrable functions on $\Sd$ with the inner product

\begin{equation*}
\left\langle f,g\right\rangle:=\int_{\Sd}f(x)g(x)\text{d}\omega_d(x)
\end{equation*}
and the induced norm $\|f\|_{\Lt}: = \sqrt{\left\langle f,f\right\rangle}$. By $\mathcal{C}(\Sd)$ we denote the space of continuous functions on $\Sd$, endowed with the uniform norm $\|f\|_{\infty}:=\sup_{x\in\Sd}|f(x)|$.

The restriction to $\Sd$ of a homogeneous and harmonic polynomial of total degree $\ell$ defined on $\mathbb{R}^{d+1}$ is called a \emph{spherical harmonic of degree $\ell$} on $\Sd$. We denote, as usual, by $\{Y_{\ell,k}:k = 1,2,\ldots,Z(d,\ell)\}$ a collection of $\Lt$-orthonormal real-valued spherical harmonics of exact degree $\ell$, where
\begin{equation}\label{equ:numberZ}
Z(d,0) = 1, \quad Z(d,\ell) = (2\ell +d-1)\frac{\Gamma(\ell+d-1)}{\Gamma(d)\Gamma(\ell+1)}\sim \frac{2}{\Gamma(d)}\ell^{d-1}\quad\text{as }\ell\rightarrow\infty,
\end{equation}
where $\Gamma(z)$ is the gamma function and $f(x)\sim g(x)$ as $x\rightarrow c$ means $f(x)/g(x)\rightarrow 1$ as $x\rightarrow c$. The spherical harmonics of degree $\ell\in \{0,1,2,\ldots\}$ satisfy the addition theorem \cite[Theorem 2]{MR0199449}, that is,
\begin{equation*}
\sum_{k=1}^{Z(d,\ell)}Y_{\ell,k}(x)Y_{\ell,k}(y)= \frac{Z(d,\ell)}{\lvert\Sd\rvert}P_{\ell}^{(d)}(x\cdot y),
\end{equation*}
where $P_{\ell}^{(d)}$ is the normalized Gegenbauer polynomial on $[-1,1]$, orthogonal on with respect to the weight function $(1-t^2)^{d/2-1}$, and normalized such that $P^{(d)}_{\ell}(1)=1$. As an immediate application of the addition theorem, we have
\begin{equation}\label{equ:sphericalharmincsbound}
\|Y_{\ell,k}\|_{\infty}\leq \left(Z(d,\ell)/\lvert\Sd\rvert\right)^{1/2}\quad \forall \ell=0,1,2,\ldots\text{ and }k=1,2,\ldots,Z(d,\ell).
\end{equation}
Indeed, for any spherical harmonic $Y_{\ell,k}$, suppose $\lvert Y_{\ell,k}(x)\rvert$ attains $\|Y_{\ell,k}\|_{\infty}$ at the point $x^*\in\Sd$, then
\begin{equation*}
\|Y_{\ell,k}\|_{\infty} = \lvert Y_{\ell,k}(x^*)\rvert \leq \left(\sum_{k=1}^{Z(d,\ell)}\lvert Y_{\ell,k}(x^*)\rvert^2\right)^{1/2} = (Z(d,\ell)P^{(d)}_{\ell}(1)/\lvert\Sd\rvert)^{1/2} = \left(Z(d,\ell)/ \lvert\Sd\rvert\right)^{1/2}.
\end{equation*}
Besides, it is well known (see, e.g., \cite[pp. 38–39]{MR0199449}) that each spherical harmonic $Y_{\ell,k}$ of exact degree $\ell$ is an eigenfunction of the negative Laplace--Beltrami operator $-\Delta^*_d$ for $\Sd$ with eigenvalue 
\begin{equation}\label{equ:LBeigenvale}
\lambda_{\ell}:=\ell(\ell+d-1).
\end{equation}

The family $\{Y_{\ell,k}:k=1,\ldots,Z(d,\ell);\ell = 0,1,2,\ldots\}$ of spherical harmonics forms a complete $\Lt$-orthonormal (with respect to $\omega_d$) system for the Hilbert space $\Lt(\Sd)$. Thus, for any $f\in\Lt(\Sd)$, it can be represented by a Laplace--Fourier series
\begin{equation*}
f(x)=\sum_{\ell=0}^{\infty}\sum_{k=1}^{Z(d,\ell)}\hat{f}_{\ell,k} Y_{\ell,k}(x)
\end{equation*}
with coefficients
\begin{equation}\label{equ:LFcoefficient}
\hat{f}_{\ell,k}:=\left\langle f,Y_{\ell,k}\right\rangle=\int_{\Sd}f(x)Y_{\ell,k}(x)\text{d}\omega_d(x),\quad \ell=0,1,2,\ldots \text{ and }k = 1,2,\ldots,Z(d,\ell).
\end{equation}

The space $\Pn(\Sd)$ of all spherical polynomials of degree at most $n$ (i.e., the restriction to $\Sd$ of all polynomials in $\mathbb{R}^{d+1}$ of degree at most $n$) coincides with the span of all spherical harmonics up to (and including) degree $n$, and its dimension satisfies $\dim(\Pn(\Sd))=Z(d+1,n)$. The space $\Pn(\Sd)$ is also a reproducing kernel Hilbert space with the reproducing kernel
\begin{equation}\label{equ:kernel}
G_n(x,y) = \sum_{\ell=0}^n\sum_{k=1}^{Z(d,\ell)}Y_{\ell,k}(x)Y_{\ell,k}(y)
\end{equation}
in the sense that 
\begin{equation}\label{equ:reproducingproperty}
\left\langle \chi,G(\cdot,x) \right\rangle = \chi(x)\quad\forall \chi\in\Pn(\Sd),
\end{equation}
see, e.g., \cite{MR1115901}. Given $f\in \mathcal{C}(\Sd)$, it is often simpler in practice to express the hyperinterpolant $\mathcal{L}_nf$ using the reproducing kernel $G_n(\cdot,\cdot)$ defined by \eqref{equ:kernel}. By rearranging the summation, 
\begin{equation*}
\mathcal{L}_nf(x) = \sum_{\ell=0}^{n}\sum_{k=1}^{Z(d,\ell)}\left(\sum_{j=1}^mw_j f(x_j)Y_{\ell,k}(x_j)\right)Y_{\ell,k}(x) = \sum_{j=1}^mw_jf(x_j)G_n(x,x_j).
\end{equation*}
Since such a summation-rearranging procedure does not depend on the quadrature exactness, such an expression also applies to $\U_nf$ and $\mathcal{Q}_nf$. What makes the above three expressions different is the quadrature rules used for constructing different kinds of hyperinterpolants.

\subsection{Sobolev spaces}
The study of hyperinterpolation in a Sobolev space setting can be traced back to the work \cite{MR2274179} by Hesse and Sloan. The Sobolev space $\Hs(\Sd)$ on the sphere $\Sd$ may be defined for $s\geq 0$ as the set of all functions $f\in\Lt(\Sd)$ whose Laplace--Fourier coefficients \eqref{equ:LFcoefficient} satisfy
\begin{equation*}
\sum_{\ell=0}^{\infty}\sum_{k=1}^{Z(d,\ell)}(1+\lambda_{\ell})^s\lvert \hat{f}_{\ell,k}\rvert^2<\infty,
\end{equation*}
where $\lambda_{\ell}$ is given as \eqref{equ:LBeigenvale}. When $s=0$, we have $H^0(\Sd)=\Lt(\Sd)$. The norm in $\Hs(\Sd)$ may be defined as the square root of the expression on the left-hand side of the last inequality; however, in this paper, we shall take advantage of the freedom to define equivalent Sobolev space norms. Let $s>d/2$ be fixed and suppose we are given a sequence of positive real numbers $(a^{(s)}_{\ell})_{\ell\geq 0}$ satisfying 
\begin{equation}\label{equ:asl}
a^{(s)}_{\ell} \asymp (1+\lambda_{\ell})^{-s} \asymp (1+\ell)^{-2s},
\end{equation}
where $a_n\asymp b_n$ denotes that there exist $c_1,c_2>0$ independent of $n$ such that $c_1a_n\leq b_n\leq c_2b_n$. Then we can define a norm in $\Hs(\Sd)$ by
\begin{equation*}
\|f\|_{\Hs}:=\left(\sum_{\ell=0}^{\infty}\sum_{k=1}^{Z(d,\ell)}\frac{1}{a^{(s)}_{\ell}}\lvert \hat{f}_{\ell,k}\rvert^2\right)^{1/2}.
\end{equation*}
The norm $\|\cdot\|_{\Hs}$ therefore depends on the particular choice of the sequence $(a^{(s)}_{\ell})_{\ell\geq 0}$, but a change to this sequence merely leads to an \emph{equivalent} Sobolev norm.

The following lemmas are necessary for our analysis.

\begin{lemma}\label{lem:hsl2}
For any $f\in\Pn(\Sd)$, $\|f\|_{\Hs}\leq \tilde{c} \left(n+1\right)^s\|f\|_{\Lt}$, where $\tilde{c}>0$ is a constant. 
\end{lemma}
\begin{proof}
It is straightforward that
\begin{equation*}
\|f\|_{\Hs}=\left(\sum_{\ell=0}^{n}\sum_{k=1}^{Z(d,\ell)}\frac{1}{a^{(s)}_{\ell}}\lvert\hat{f}_{\ell,k}\rvert^2\right)^{1/2}\leq \left(\frac{1}{a_n^{(s)}}\|f\|_{\Lt}^2\right)^{1/2}\leq \tilde{c} \left(n+1\right)^s\|f\|_{\Lt}\quad \forall f\in\Pn(\Sd),
\end{equation*}
where we used the order \eqref{equ:asl} of $(a^{(s)}_{\ell})_{\ell\geq 0}$. 
\end{proof}
\begin{lemma}\label{lem:sobolevclosed}
If $s>d/2$, then $\|fg\|_{\Hs}\leq       \check{c}\|f\|_{\Hs}\|g\|_{\Hs}$, where $\check{c}>0$ is some constant. 
\end{lemma}
\begin{proof}
For any Lipschitz domain $\Omega$, let $W^{s,2}(\Omega)$ be the Sobolev space of those functions in $L^2(\Omega)$ whose distributional derivatives up to (and including) order $s$ are in $L^2(\Omega)$. Note that the Sobolev spaces $\Hs(\Sd)$ can also be defined with the help of charts (that is, the so-called Sobolev spaces over boundaries), giving the space $W^{s,2}(\Sd)$ with an \emph{equivalent} norm, that is,
\begin{equation}\label{equ:sobolevlem1}
c_1\|f\|_{\Hs}\leq \|f\|_{W^{s,2}(\Sd)}\leq c_2\|f\|_{\Hs},
\end{equation}
where $c_1,c_2>0$ are some constants; see \cite[Chapter 7.3]{MR0350177} or \cite[Chapter 7.2.3]{MR2511061}. If $s>d/2$, then the Sobolev space $W^{s,2}(\Sd)$ is a Banach algebra, that is, for any $f,g\in W^{s,2}(\Sd)$,
\begin{equation}\label{equ:sobolevlem2}
\|fg\|_{W^{s,2}(\Sd)}  \leq c_3 \|f\|_{W^{s,2}(\Sd)}\|g\|_{W^{s,2}(\Sd)},
\end{equation}
where $c_3>0$ is some constant; we refer to \cite[Theorem 5.23]{MR0450957} or \cite[Section 6.1]{MR785568} for this result. Together with \eqref{equ:sobolevlem1} and \eqref{equ:sobolevlem2}, we have the desired estimate.
\end{proof}
\begin{remark}
The norm equivalence \eqref{equ:sobolevlem1} is also identified and utilized in some other spherical approximation schemes, see, e.g., \cite{MR3712286,MR2271729}.
\end{remark}

\subsection{Spherical $t$-designs and QMC designs}\label{sec:designs}

A spherical $t$-design, introduced in the remarkable paper \cite{delsarte1991geometriae} by Delsarte, Goethals, and Seidel, is a set of points $\{x_j\}_{j=1}^m\subset \Sd$ with the characterizing property that an equal-weight quadrature rule in these points exactly integrates all polynomials of degree at most $t$, that is,
\begin{equation}\label{equ:stdexactness}
\frac{1}{m}\sum_{j=1}^m\chi(x_j)=\int_{\Sd}\chi(x)\text{d}\omega_d(x)\quad\forall \chi\in\mathbb{P}_t.
\end{equation}

A majority of studies in the literature on spherical designs care about the relation between $m$ and $t$ in \eqref{equ:stdexactness}. It was known by Seymour and Zaslavsky \cite{MR744857} that a spherical $t$-design always exists if $m$ is sufficiently large, but no quantitative results on the size of $m$ were established. In the original manuscript \cite{delsarte1991geometriae} of spherical $t$-designs, lower bounds on $m$ of exact order $t^d$ were derived in the sense that
\begin{equation*}
       m\geq\begin{dcases}
\binom{d+t/2}{d}+\binom{d+t/2-1}{d}&\text{for even }t,\\
2\binom{d+\lfloor t/2\rfloor}{d}&\text{for odd }t;
       \end{dcases}
\end{equation*}
but according to Bannai and Damerell \cite{MR519045,MR576179}, the number $m$ of quadrature points could achieve these lower bounds only for a few small values of $t$. Bondarenko, Radchenko, and Viazovska asserted in \cite{MR3071504} that for each $m\geq ct^d$ with some positive but unknown constant $c>0$, there exists a spherical $t$-design in $\Sd$ consisting of $m$ points.

Quadrature rules \eqref{equ:quad} using spherical $t$-designs are known to have fast-convergence property when the integrand belongs to the Sobolev space $\Hs$; namely, given $s>d/2$, there exists $C(s,d)>0$ depending only on $s$ and $d$ such that for every $m$-point spherical $t$-design $\{x_j\}_{j=1}^m$ on $\Sd$, there holds 
\begin{equation}\label{equ:stderror}
\sup_{\substack{f\in\Hs(\Sd),\\ \|f\|_{\Hs}\leq 1}} \left\lvert\frac{1}{m}\sum_{j=1}^mf(x_j)-\int_{\Sd}f(x)\text{d}\omega_d\right\rvert\leq\frac{C(s,d)}{t^s}.
\end{equation}
The estimate \eqref{equ:stderror} was established gradually: It was first proved for the particular case $s=3/2$ and $d=2$ in \cite{MR2127668}, then extended to all $s>1$ for $d=2$ in \cite{MR2252093}, and finally extended to all $s>d/2$ and all $d\geq 2$ in \cite{MR2263736}. The condition $s>d/2$ is a natural one because functions to be approximated in this paper are assumed to be continuous, and by the Sobolev embedding theorem, $\Hs(\Sd)$ is continuously embedded in $\mathcal{C}(\Sd)$ if $s>d/2$.

If only spherical $t$-designs with $m\asymp t^d$ are concerned, then the upper bound on the error \eqref{equ:stderror} is of order $m^{-s/d}$. Here comes the concept of QMC designs, introduced by Brauchart, Saff, Sloan, and Womersley in \cite{MR3246811}: Given $s>d/2$, a sequence $\{x_j\}_{j=1}^m$ of $m$-point configurations on $\Sd$ with $m\rightarrow\infty$ is said to be a sequence of \emph{QMC designs} for $\Hs(\Sd)$ if there exists $c(s,d)>0$ independent of $m$ such that
\begin{equation}\label{equ:QMCerror}
\sup_{\substack{f\in\Hs(\Sd),\\ \|f\|_{\Hs}\leq 1}} \left\lvert\frac{1}{m}\sum_{j=1}^mf(x_j)-\int_{\Sd}f(x)\text{d}\omega_d\right\rvert\leq\frac{c(s,d)}{m^{s/d}}.
\end{equation}
In a nutshell, quadrature rules using QMC designs provide the same asymptotic order of convergence as exact rules (e.g., rules using spherical $t$-designs) when the integrand belongs to the Sobolev space $H^s$, but are easier to obtain numerically. For more studies on the numerical integration on the sphere with the integrand belonging to a Sobolev space, we refer the reader to \cite{MR2929076,MR3365840,hesse2010numerical,MR2123223}. Equal-weight numerical integration rules with the integrand belonging to many other spaces of smoothness also attracts much interest, see, e.g., \cite{MR2558691,MR2391005,MR3038697,MR4438170,MR3614894,MR2059741,MR3530965,MR1617765}, to name a few.

A substantial definition related to QMC designs $\{x_j\}_{j=1}^m$ is the \emph{QMC strength}, denoted by $s^*$. For every sequence of QMC designs $\{x_j\}_{j=1}^m$, there is some number $s^*$ such that $\{x_j\}_{j=1}^m$ is a sequence of QMC designs for all $s$ satisfying $d/2<s\leq s^*$ and is not a QMC design for $s>s^*$. Even if the integrand $f$ is infinitely differentiable, the convergence rate of the numerical integration error \eqref{equ:QMCerror} using a QMC design with strength $s^*$ is controlled by $m^{-s^*/d}$. 

\section{General framework of unfettered hyperinterpolation}\label{sec:unfetteredtheory}

With the aid of the reproducing property \eqref{equ:reproducingproperty}, the Marcinkiewicz--Zygmund property \eqref{equ:etaassumption} implies the following lemma.

\begin{lemma}\label{lem}
For any $\chi\in\Pn(\Sd)$, we have

{\rm{(a)}} $(1-\eta)\|\chi\|_{\Lt}^2\leq \left\langle \U_n\chi,\chi\right\rangle\leq(1+\eta)\|\chi\|_{\Lt}^2.$

{\rm{(b)}} $(1-\eta)\|\chi\|_{\Lt}\leq\|\U_n\chi\|_{\Lt}\leq(1+\eta)\|\chi\|_{\Lt}$.

{\rm{(c)}} $\|\U_n\chi-\chi\|_{\Lt}^2\leq(\eta^2+4\eta)\|\chi\|_{\Lt}^2.$
\end{lemma} 

\begin{proof} (a) The reproducing property \eqref{equ:reproducingproperty} of $G_n(\cdot,\cdot)$ implies
       \begin{equation*}\begin{split}
       \left\langle \U_n\chi,\chi\right\rangle &= \left\langle \sum_{j=1}^mw_j\chi(x_j)G_n(x,x_j),\chi(x)\right\rangle=\sum_{j=1}^mw_j\chi(x_j)\left\langle G_n(x,x_j),\chi(x)\right\rangle=\sum_{j=1}^nw_j\chi(x_j)^2.
       \end{split}\end{equation*}
       Thus by the Marcinkiewicz--Zygmund property \eqref{equ:etaassumption},
       \begin{equation*}
       (1-\eta)\|\chi\|_{\Lt}^2=(1-\eta)\int_{\Sd}\chi^2\text{d}\omega_d\leq\sum_{j=1}^nw_j\chi(x_j)^2\leq(1+\eta)\int_{\Sd}\chi^2\text{d}\omega_d=(1+\eta)\|\chi\|_{\Lt}^2.
       \end{equation*}

       (b) By part (a), we have $(1-\eta)\|\chi\|_{\Lt}^2\leq\left\langle \U_n\chi,\chi\right\rangle\leq\|\U_n\chi\|_{\Lt}\|\chi\|_{\Lt}$, leading to $(1-\eta)\|\chi\|_{\Lt}\leq\|\U_n\chi\|_{\Lt}$. We also have
              \begin{equation*}\begin{split}
              \|\U_n\chi\|_{\Lt}^2&\leq\left\langle \U_n\chi,\U_n\chi\right\rangle
              =\left\langle \sum_{j=1}^mw_j\chi(x_j)G_n(x,x_j),\U_n\chi(x)\right\rangle=\sum_{j=1}^mw_j\chi(x_j)\U_n\chi(x_j)\\
              &\leq \left(\sum_{j=1}^mw_j\chi(x_j)^2\right)^{1/2}\left(\sum_{j=1}^mw_j\left(\U_n\chi(x_j)\right)^2\right)^{1/2}
              \leq (1+\eta)\|\chi\|_{\Lt}\|\U_n\chi\|_{\Lt},
              \end{split}\end{equation*}
              where the first inequality is due to the Cauchy--Schwarz inequality, and the second one is ensured by the Marcinkiewicz--Zygmund property \eqref{equ:etaassumption}. Thus part (b) is proved.

       (c) Using parts (a) and (b) above, it is straightforward that
       \begin{equation*}\begin{split}
       \|\U_n\chi-\chi\|_{\Lt}^2
       &=\|\U_n\chi\|_{\Lt}^2-2\left\langle \U_n\chi,\chi\right\rangle+\|\chi\|_{\Lt}^2
       \leq (1+\eta)^2\|\chi\|_{\Lt}^2-2(1-\eta)\|\chi\|_{\Lt}^2+\|\chi\|_{\Lt}^2\\
       &=(\eta^2+4\eta)\|\chi\|_{\Lt}^2.
       \end{split}\end{equation*}
Hence this lemma is proved.
\end{proof}

We are now ready to state our main theorem.

\begin{theorem}\label{thm}
Given $f\in \mathcal{C}(\Sd)$, let $\U_nf\in\mathbb{P}_n$ be its unfettered hyperinterpolant defined by \eqref{equ:unfetteredhyper}, where the $m$-point positive-weight quadrature rule \eqref{equ:quad} is only assumed to have the Marcinkiewicz--Zygmund property \eqref{equ:etaassumption} with $\eta\in[0,1)$. Then
\begin{equation}\label{equ:stability}
\|\U_nf\|_{\Lt}\leq\sqrt{1+\eta}\left(\sum_{j=1}^mw_j\right)^{1/2}\|f\|_{\infty},
\end{equation}
and
\begin{equation}\label{equ:error}
\|\U_nf-f\|_{\Lt}\leq \left(\sqrt{1+\eta}\left(\sum_{j=1}^mw_j\right)^{1/2}+\lvert\Sd\rvert^{1/2}\right)E_n(f)+\sqrt{\eta^2+4\eta}\|\chi^*\|_{\Lt},
\end{equation}
where $E_n(f)$ denotes the best uniform error of $f$ by a polynomial in $\Pn(\Sd)$ and $\chi^*\in\Pn(\Sd)$ denotes the best approximation polynomial of $f$ in $\Pn(\Sd)$ in the sense of $\|f-\chi^*\|_{\infty}=E_n(f)$. 
\end{theorem}
\begin{proof}
For any $f\in \mathcal{C}(\Sd)$, we have $\U_nf\in\mathbb{P}_n$ and hence $\left\langle G_n(x,x_j),\U_nf(x)\right\rangle=\U_nf(x_j)$. Thus, 
\begin{equation*}\begin{split}
\left\langle \U_nf,\U_nf\right\rangle
& = \left\langle \sum_{j=1}^mw_jf(x_j)G_n(x,x_j),\U_nf(x)\right\rangle = \sum_{j=1}^mw_jf(x_j)\U_nf(x_j)\\
& \leq \left(\sum_{j=1}^mw_jf(x_j)^2\right)^{1/2}\left(\sum_{j=1}^mw_j\left(\U_n\chi(x_j)\right)^2\right)^{1/2}\leq \left(\sum_{j=1}^mw_j\right)^{1/2}\|f\|_{\infty}\sqrt{1+\eta}\|\U_nf\|_{\Lt},
\end{split}\end{equation*}
where the first inequality is due to the Cauchy--Schwarz inequality and the second one holds by using $\sum_{j=1}^mw_jf(x_j)^2\leq\|f\|_{\infty}^2\sum_{j=1}^mw_j$ and the Marcinkiewicz--Zygmund property \eqref{equ:etaassumption}. This estimate immediately implies the stability result \eqref{equ:stability}.

The error bound \eqref{equ:error} is obtained by the following argument. For any $\chi\in\mathbb{P}_n$, we have
\begin{equation*}\begin{split}
\|\U_nf-f\|_{\Lt}
& = \|\U_n(f-\chi)+(\chi-f)+(\U_n\chi-\chi)\|_{\Lt} \leq \|\U_n(f-\chi)\|_{\Lt} + \|f-\chi\|_{\Lt}+\|\U_n\chi-\chi\|_{\Lt}\\
& \leq \sqrt{1+\eta}\left(\sum_{j=1}^mw_j\right)^{1/2}\|f-\chi\|_{\infty}+\lvert\Sd\rvert^{1/2}\|f-\chi\|_{\infty} + \|\U_n\chi-\chi\|_{\Lt}.
\end{split}\end{equation*}
It follows, since this estimate holds for all polynomials in $\Pn(\Sd)$, that
\begin{equation*}
\|\U_nf-f\|_{\Lt} \leq  \left(\sqrt{1+\eta}\left(\sum_{j=1}^mw_j\right)^{1/2}+\lvert\Sd\rvert^{1/2}\right)E_n(f) + \|\U_n\chi^*-\chi^*\|_{\Lt}.
\end{equation*}
By part (c) of Lemma \ref{lem}, we have $\|\U_n\chi^*-\chi^*\|_{\Lt}\leq \sqrt{\eta^2+4\eta}\|\chi^*\|_{\Lt}$.
\end{proof}

\subsection{Connections in the literature}

If the quadrature rule \eqref{equ:quad} is additional assumed to integral all constant functions (polynomials of degree zero) exactly, that is, $\sum_{j=1}^mw_j=\lvert\Sd\rvert$,
then we have $\|\U_nf\|_{\Lt}\leq\sqrt{1+\eta}\lvert\Sd\rvert^{1/2}\|f\|_{\infty}$ and 
\begin{equation*}
\|\U_nf-f\|_{\Lt}\leq \left(\sqrt{1+\eta}+1\right)\lvert\Sd\rvert^{1/2}E_n(f)+\sqrt{\eta^2+4\eta}\|\chi^*\|_{\Lt}.
\end{equation*}

If the quadrature rule \eqref{equ:quad} exactly integrate all polynomials of degree at most $2n$, i.e., the constant $\eta$ is zero, then the stability result \eqref{equ:stability} and error bound \eqref{equ:error} reduce to the classical results of hyperinterpolation in \cite{sloan1995polynomial}; namely, $\|\U_nf\|_{\Lt}\leq \lvert\Sd\rvert^{1/2}\|f\|_{\infty}$ and 
\begin{equation*}
\|\U_nf-f\|_{\Lt}\leq 2\lvert\Sd\rvert^{1/2}E_n(f).
\end{equation*}

If the quadrature rule \eqref{equ:quad} has exactness degree $n+n'$ with $0<n'\leq n$, then $\U_n\chi = \chi$ for all $\chi\in\mathbb{P}_{n'}(\Sd)$, see \cite[Lemma 2.1]{an2022quadrature}. By the stability result \eqref{equ:stability}, we have for any $\chi\in\mathbb{P}_{n'}(\Sd)$,
\begin{equation*}
\|\U_nf-f\|_{\Lt} \leq\|\U_n(f-\chi) - (f-\chi)\|_{\Lt} \leq \|\U_n(f-\chi)\|_{\Lt}+\|f-\chi\|_{\Lt}.
\end{equation*}
As this estimate holds for all $\chi\in\mathbb{P}_{n'}(\Sd)$, it is straightforward that
\begin{equation}\label{equ:estimatenn'}
\|\U_nf-f\|_{\Lt} \leq \left(\sqrt{1+\eta}+1\right)\lvert\Sd\rvert^{1/2}E_{n'}(f),
\end{equation}
which has the same convergence rate in terms of $E_{n'}(f)$ as our previous estimate \eqref{equ:BIT} in \cite{an2022quadrature}. In \cite{an2022quadrature}, we make use of the discrete orthogonal projection property (see \cite[Lemma 3.1]{an2022quadrature}) to obtain the estimate \eqref{equ:BIT}, while in this paper we utilize the reproducing property \eqref{equ:reproducingproperty} for the estimate \eqref{equ:estimatenn'}.

Moreover, in light of Theorem \ref{thm} and the study on spherical hyperinterpolation in a Sobolev space setting by Hesse and Sloan in \cite{MR2274179}, we have the following Sobolev estimates, which reduce to their results in \cite{MR2274179} when the exactness degree $2n$ is assumed. For simplicity and without loss of generality, we assume $\sum_{j=1}^mw_j=\lvert\Sd\rvert$ in Corollary \ref{cor:sobolev}. Note that $\Hs(\Sd)\subset \Lt(\Sd)$.
\begin{corollary}\label{cor:sobolev}
Let $d\geq 2$, and let $t$ and $s$ be fixed real numbers with $s\geq t \geq 0$ and $s\geq d/2$. Under the conditions of Theorem \ref{thm}, for any unfettered hyperinterpolation operator $\U_n:\Hs(\Sd)\rightarrow H^t(\Sd)$, there hold
\begin{equation}\label{equ:errorsob}
\|\U_nf\|_{H^t} \leq \tilde{c} \left[ \left(\sqrt{1+\eta}\lvert\Sd\rvert^{1/2}+1\right)(n+1)^{d/2+t-s}\|f\|_{H^s} +  (n+1)^t\sqrt{\eta^2+4\eta}\|f\|_{\Lt}\right] + \|f\|_{\Hs}
\end{equation}
and
\begin{equation}\label{equ:stabilitysob}
\|\U_nf-f\|_{H^t} \leq \tilde{c} \left[ \left(\sqrt{1+\eta}\lvert\Sd\rvert^{1/2}+1\right)(n+1)^{d/2+t-s} E_n(f;\Hs(\Sd)) +  \tilde{c}(n+1)^t\sqrt{\eta^2+4\eta}\|f\|_{\Lt}\right],
\end{equation}
where $\tilde{c}>0$ is some constant that may vary line to line, and $E_n(f;\Hs(\Sd))$ is the best $\Hs$ approximation of $f\in\Hs(\Sd)$ by a polynomial in $\Pn(\Sd)$, that is, $E_n(f;\Hs(\Sd)):=\inf_{\chi\in\Pn(\Sd)}\|f-\chi\|_{\Hs}$.
\end{corollary}

\begin{remark}
 When the exactness degree of the rule \eqref{equ:quad} is assumed to be $2n$, $\eta =0$ and the results \eqref{equ:stabilitysob} and \eqref{equ:errorsob} reduce to the respective results of the original hyperinterpolation (some constants may be different) derived by Hesse and Sloan in \cite{MR2274179}.
 \end{remark}

\begin{proof}
Similar to the decomposition of $\|\U_nf-f\|_{\Lt}$ in the proof of Theorem \ref{thm}, we have
\begin{equation}\label{equ:sobolevdecomp}
\|\U_nf-f\|_{H^t} \leq \|\U_n(f-\mathcal{P}_nf)\|_{H^t} + \|\mathcal{P}_nf-f\|_{H^t} + \|\U_n(\mathcal{P}_nf)-\mathcal{P}_nf\|_{H^t}.
\end{equation}
The first term on the right-hand side of \eqref{equ:sobolevdecomp} can be bounded by
\begin{equation*}\begin{split}
\|\U_n(f-\mathcal{P}_nf)\|_{H^t} &\leq \tilde{c}(n+1)^t\|\U_n(f-\mathcal{P}_nf)\|_{\Lt} \leq \tilde{c}(n+1)^t\sqrt{1+\eta}\lvert\Sd\rvert^{1/2}\|f-\mathcal{P}_nf\|_{\infty}\\
& \leq \tilde{c} (n+1)^t\sqrt{1+\eta}\lvert\Sd\rvert^{1/2} (n+1)^{d/2-s} \|f-\mathcal{P}_nf\|_{\Hs},
\end{split}
\end{equation*}
where the first inequality is due to Lemma \ref{lem:hsl2}, the second is due to the stability result \eqref{equ:stability}, and the third is due to \cite[Lemma 3.5]{MR2274179}. This lemma also guarantees that
\begin{equation*}
\|\mathcal{P}_nf-f\|_{H^t}  \leq \tilde{c} (n+1)^{t-s}\|\mathcal{P}_nf-f\|_{H^s}.
\end{equation*}
The third term can be estimated as 
\begin{equation*}\begin{split}
\|\U_n(\mathcal{P}_nf)-\mathcal{P}_nf\|_{H^t}
& \leq \tilde{c}(n+1)^t\|\U_n(\mathcal{P}_nf)-\mathcal{P}_nf\|_{\Lt} \leq \tilde{c}(n+1)^t\sqrt{\eta^2+4\eta}\|\mathcal{P}_nf\|_{\Lt}\\
&  \leq \tilde{c}(n+1)^t\sqrt{\eta^2+4\eta}\|f\|_{\Lt}
\end{split}
\end{equation*}
where the first inequality is due to Lemma \ref{lem:hsl2}, the second is due to part (c) of Lemma \ref{lem}, and the third is due to the fact that the norm of $\mathcal{P}_n$ as an operator from $\Lt(\Sd)$ onto $\Lt(\Sd)$ is 1. Thus we have 
\begin{equation*}
\|\U_nf-f\|_{H^t} \leq \tilde{c} \left[ \left(\sqrt{1+\eta}\lvert\Sd\rvert^{1/2}+1\right)(n+1)^{d/2+t-s} E_n(f;\Hs(\Sd)) +  (n+1)^t\sqrt{\eta^2+4\eta}\|f\|_{\Lt}\right],
\end{equation*}
where $E_n(f;\Hs(\Sd)) = \|f-\mathcal{P}_nf\|_{\Hs}$ is verified by \cite[Equ. (3.22)]{MR2274179}.

As $\|f-\mathcal{P}_nf\|_{\Hs}\leq \|f\|_{H^s}$ and $\|f\|_{H^t}\leq \|f\|_{H^s}$, we have
\begin{equation*}\begin{split}
\|\U_nf\|_{H^t}& \leq \|\U_nf-f\|_{H^t}  + \|f\|_{H^t}\\
&\leq \tilde{c} \left[ \left(\sqrt{1+\eta}\lvert\Sd\rvert^{1/2}+1\right)(n+1)^{d/2+t-s}\|f\|_{H^s} +  (n+1)^t\sqrt{\eta^2+4\eta}\|f\|_{\Lt}\right] + \|f\|_{\Hs},
\end{split}\end{equation*}
which completes the proof of this corollary.
\end{proof}

\subsection{Scattered data}

Now together with the work \cite{MR2475947} of Le Gia and Mhaskar, we can obtain a probabilistic description of Theorem \ref{thm}.
\begin{lemma}[{\cite[p. 463]{MR2475947}}]\label{lem:legiamhaskar}
Let the quadrature rule for constructing the unfettered hyperinterplants be an equal-weight rule \eqref{equ:equalweightquad} with an independent random sample of $m$ points drawn from the distribution $\omega_d$, and let $\gamma>0$ and $\eta\in(0,1)$. Then there exists a constant $\bar{c}:=\bar{c}(\gamma)$ such that if $m\geq \bar{c} n ^d\log{n}/\eta^2$, then the Marcinkiewicz--Zygmund property \eqref{equ:etaassumption} holds with probability exceeding $1-\bar{c}n^{-\gamma}$.
\end{lemma}

\begin{corollary}\label{cor:unfettered}
Adopt conditions of Theorem \ref{thm} and Lemma \ref{lem:legiamhaskar}, where the quadrature rule for constructing $\U_nf$ takes the form of \eqref{equ:equalweightquad} and uses $m\geq \bar{c}(\gamma) n ^d\log{n}/\eta^2$ quadrature points. Then the stability result \eqref{equ:stability} and error bound \eqref{equ:error} are valid with probability exceeding $1-\bar{c}n^{-\gamma}$.
\end{corollary}

As we can see, having bypassed the quadrature exactness assumption of the original hyperinterpolation, Theorem \ref{thm}
 provides a general framework of analyzing the behavior of the unfettered hyperinterpolation. What we need to do in practice is to control the constant $\eta$ occurred in the Marcinkiewicz--Zygmund property \eqref{equ:etaassumption}. As a practical guide, if the quadrature points are independently random samples from the the distribution $\omega_d$, then Corollary \ref{cor:unfettered} suggests a simple way to decrease $\eta$ by increasing the number $m$ of quadrature points.

\section{Unfettered hyperinterpolation with QMC designs}\label{sec:QMCtheory}
If $\{x_j\}_{j=1}^m$ is a QMC design for $\Hs(\Sd)$, it can be managed to satisfy the Marcinkiewicz--Zygmund property \eqref{equ:etaassumption}, as shown in Section \ref{subsec:QMCinferior}. Hence the unfettered hyperinterpolation using QMC designs is a special case of the general framework analyzed in Theorem \ref{thm}. Recall that we refer to such approximation as the QMC hyperinterpolation, denoted by $\mathcal{Q}_n$. However, the obtained error estimate may not be optimal due to the generality
of Theorem \ref{thm}, and we can find a sharper estimate customized for the unfettered hyperinterpolation using QMC designs.

\subsection{QMC hyperinterpolation in the general framework of unfettered hyperinterpolation}\label{subsec:QMCinferior}
It is critical to note that the numerical integration error \eqref{equ:QMCerror} of the QMC design-based quadrature rule and the Marcinkiewicz--Zygmund property \eqref{equ:etaassumption} cannot imply each other. On the one hand, the error \eqref{equ:QMCerror} applies to all functions in $\Hs(\Sd)$ with the property \eqref{equ:etaassumption} only holds for polynomial $\chi^2$ with $\chi\in\Pn(\Sd)$. On the other hand, if the integrand in the quadrature rule \eqref{equ:equalweightquad} is $\chi^2$ with $\chi\in\Pn(\Sd)$, the error bound \eqref{equ:QMCerror} suggests
\begin{equation}\label{equ:QMCchi2}
\left\lvert\frac{1}{m}\sum_{j=1}^m\chi(x_j)^2-\int_{\Sd}\chi^2\text{d}\omega_d\right\rvert\leq \frac{c(s,d)}{m^{s/d}}\|\chi^2\|_{\Hs}.
\end{equation}
This error \eqref{equ:QMCchi2} is not compatible with the Marcinkiewicz--Zygmund property \eqref{equ:etaassumption} because the controlling term is $\|\chi^2\|_{\Hs}$ instead of $\int_{\Sd}\chi^2\text{d}\omega_d$. Nevertheless, we can find an upper bound of $\|\chi^2\|_{\Hs}$ in terms of $\int_{\Sd}\chi^2\text{d}\omega_d$ to transform the error \eqref{equ:QMCchi2} into a Marcinkiewicz--Zygmund property \eqref{equ:etaassumption}. With the aid of Lemma \ref{lem:hsl2}, we have
\begin{equation*}\begin{split}
\|\chi^2\|_{\Hs}
&\leq \tilde{c}(2n+1)^s\|\chi^2\|_{\Lt}\leq \tilde{c}(2n+1)^s\|\chi\|_{\infty}\|\chi\|_{\Lt} \leq \tilde{c}(2n+1)^s\frac{\|\chi\|_{\infty}}{\|\chi\|_{\Lt}}\int_{\Sd}\chi^2\text{d}\omega_d.
\end{split}\end{equation*}
For any $\chi = \sum_{\ell=0}^n\sum_{k=1}^{Z(d,\ell)}\alpha_{\ell,k}Y_{\ell,k}\in\Pn(\Sd)$, we have 
\begin{equation*}
\frac{\|\chi\|_{\infty}}{\|\chi\|_{\Lt}}\leq \frac{\sum_{\ell=0}^n\sum_{k=1}^{Z(d,\ell)}\lvert\alpha_{\ell,k}\rvert\|Y_{\ell,k}\|_{\infty}}{\sqrt{\sum_{\ell=0}^n\sum_{k=1}^{Z(d,\ell)}\lvert\alpha_{\ell,k}\rvert^2}} \leq \sqrt{\frac{Z(d,n)}{\lvert\Sd\rvert}Z(d+1,n)},
\end{equation*}
where we used the estimate \eqref{equ:sphericalharmincsbound} on the uniform norm of $Y_{\ell,k}$ and regard $\{\alpha_{\ell,k}\}$ as a vector of size $Z(d+1,n)$. Then we can let 
\begin{equation}\label{equ:order1}
\eta = \frac{c(s,d)\tilde{c}}{m^{s/d}}(2n+1)^s\sqrt{\frac{Z(d,n)}{\lvert\Sd\rvert}Z(d+1,n)}
\end{equation}
and enforce it to be in $(0,1)$. Thus in this case, with the asymptotic result \eqref{equ:numberZ} of the size of $Z(d,\ell)$, the number $m$ should have a lower bound of order $n^{d+\frac{d^2}{s}-\frac{d}{2s}}$ as $n\rightarrow\infty$.
Moreover, regarding the term $\sqrt{\eta^2+4\eta}\|\chi^*\|_{\Lt}$ in the error estimate \eqref{equ:error} in Theorem \ref{thm}, for a fixed degree, the convergence rate of this term with respect to $m$ is $m^{-s/(2d)}$.

\subsection{Approximation theory of QMC hyperinterpolation}\label{subsec:QMC}
We then show that the QMC hyperinterpolation has a sharper error estimate than the general estimate \eqref{equ:error} in Theorem \ref{thm}.
\begin{theorem}\label{QMCthm}
Given $f\in \Hs(\Sd)\subset \Lt(\Sd)$, let $\mathcal{Q}_nf\in\mathbb{P}_n$ be its QMC hyperinterpolant defined by \eqref{equ:QMChyper}, where the $m$-point equal-weight quadrature rule \eqref{equ:equalweightquad} adopts a QMC design for $\Hs(\Sd)$ as quadrature points. Then
\begin{equation}\label{equ:stabilityqmc}
\|\mathcal{Q}_nf\|_{\Lt}\leq \|f\|_{\Lt}+\frac{c'(s,d)}{m^{s/d}}(n+1)^s\|f\|_{\Hs},
\end{equation}
where $c'(s,d)>0$ is some constant depending only on $s$ and $d$, and
\begin{equation}\label{equ:errorqmc}
\|\mathcal{Q}_nf-f\|_{\Lt}\leq c''(s,d)\left(n^{-s} +\frac{1}{m^{s/d}}\sqrt{\frac{Z(d+1,n)}{a_n^{(s)}}}\right) \|f\|_{\Hs},
\end{equation}
where $c''(s,d)>0$ is some constant depending only on $s$ and $d$. 
\end{theorem}

\begin{proof}
For $f\in \Hs(\Sd)$, we have
\begin{equation*}\begin{split}
\|\mathcal{Q}_nf\|_{\Lt}^2&=\left\langle \mathcal{Q}_nf,\mathcal{Q}_nf\right\rangle
 = \left\langle \sum_{j=1}^mw_jf(x_j)G_n(x,x_j),\mathcal{Q}_nf(x)\right\rangle = \sum_{j=1}^mw_jf(x_j)\mathcal{Q}_nf(x_j)\\
& \leq \int_{\Sd}(\mathcal{Q}_n f)f\text{d}\omega_d + \frac{c(s,d)}{m^{s/d}}\|(\mathcal{Q}_nf)f\|_{\Hs} \leq \|f\|_{\Lt}\|\mathcal{Q}_n f\|_{\Lt} + \frac{c(s,d)\check{c}}{m^{s/d}}\|f\|_{\Hs}\|\mathcal{Q}_nf\|_{\Hs}\\
& \leq \|f\|_{\Lt}\|\mathcal{Q}_n f\|_{\Lt} + \frac{c(s,d)\check{c}}{m^{s/d}}\|f\|_{\Hs}(n+1)^s\|\mathcal{Q}_nf\|_{\Lt},
\end{split}\end{equation*}
where the first inequality is due to the integration error \eqref{equ:QMCerror} using QMC designs, the second one is due to the Cauchy--Schwarz inequality and Lemma \ref{lem:sobolevclosed} with $\check{c}$ given there, and the last one is due to Lemma \ref{lem:hsl2}. Hence we have the stability result \eqref{equ:stabilityqmc}.

For the error estimate \eqref{equ:errorqmc}, we have 
\begin{equation*}
\|\mathcal{Q}_nf - f\|_{\Lt} \leq \|\mathcal{Q}_nf - \mathcal{P}_n f\|_{\Lt} + \|\mathcal{P}_nf-f\|_{\Lt}, 
\end{equation*}
where $\mathcal{P}_n$ is the $\Lt$-orthgonal projection operator \eqref{equ:proj}. For the term $ \|\mathcal{P}_nf-f\|_{\Lt}$, we have 
\begin{equation*}\begin{split}
\|\mathcal{P}_nf-f\|_{\Lt}^2
& = \sum_{\ell=n+1}^{\infty}\sum_{k=1}^{Z(d,\ell)}\lvert\langle f,Y_{\ell,k}\rangle\rvert^2 = \sum_{\ell=n+1}^{\infty}\sum_{k=1}^{Z(d,\ell)}\lvert\langle f,Y_{\ell,k}\rangle\rvert^2\frac{a_{\ell}^{(s)}}{a_{\ell}^{(s)}} \leq a_n^{(s)} \|f\|_{\Hs}^2 \lesssim n^{-2s}\|f\|_{\Hs}^2.
\end{split}\end{equation*}
For the term $\|\mathcal{Q}_nf - \mathcal{P}_n f\|_{\Lt}$, we have
\begin{equation*}
\|\mathcal{Q}_nf - \mathcal{P}_n f\|_{\Lt}^2 = \sum_{\ell=0}^n\sum_{k=1}^{Z(d,\ell)}\left\lvert\left\langle f,Y_{\ell,k}\right\rangle_m - \left\langle f,Y_{\ell,k}\right\rangle \right\rvert^2
\end{equation*}
and
\begin{equation*}
\left\lvert\left\langle f,Y_{\ell,k}\right\rangle_m - \left\langle f,Y_{\ell,k}\right\rangle \right\rvert^2\leq \left(\frac{c(s,d)}{m^{s/d}}\|fY_{\ell,k}\|_{\Hs}\right)^2\leq \left(\frac{c(s,d)\check{c}}{m^{s/d}}\|f\|_{\Hs}\|Y_{\ell,k}\|_{\Hs}\right)^2,
\end{equation*}
where the first inequality is described by the integration error \eqref{equ:QMCerror} using QMC designs, and the second is due to Lemma \ref{lem:sobolevclosed}. Note that 
\begin{equation*}
\|Y_{\ell,k}\|_{\Hs}^2 = \sum_{\ell'=0}^n\sum_{k'=1}^{Z(d,\ell)}\frac{1}{a_{\ell}^{(s)}}\lvert\left\langle Y_{\ell,k},Y_{\ell',k'}\right\rangle\rvert^2 = \frac{1}{a_{\ell}^{(s)}}.
\end{equation*}
Thus
\begin{equation*}
\|\mathcal{Q}_nf - \mathcal{P}_n f\|_{\Lt}^2 \leq \left(\frac{c(s,d)\check{c}}{m^{s/d}}\|f\|_{\Hs}\right)^2 \frac{1}{a_n^{(s)}}\sum_{\ell=0}^n\sum_{k=1}^{Z(d,\ell)}1 =\left(\frac{c(s,d)\check{c}}{m^{s/d}}\|f\|_{\Hs}\right)^2 \frac{Z(d+1,n)}{a_n^{(s)}},
\end{equation*}
leading to the error estimate \eqref{equ:errorqmc}.
\end{proof}

The estimate \eqref{equ:errorqmc} consists of two terms, one representing the error of the original hyperinterpolation, and the other is newly introduced in terms of $m$. In addition to hyperinterpolation, the fully discrete needlet approximation \cite{MR3668040} using spherical needlets \cite{MR2253732,MR2237162} and using quadrature rules without exactness assumption also has error estimates of this type, see a recent contribution \cite{brauchart2022needlets}.

\begin{corollary}\label{cor:QMCpoly}
If $f\in \Hs(\Sd)\cap \Pn(\Sd)$, then $ \|\mathcal{P}_nf-f\|_{\Lt} = 0$ and
\begin{equation*}
\|\mathcal{Q}_nf-f\|_{\Lt}\leq \frac{c(s,d)\check{c}}{m^{s/d}}\sqrt{\frac{Z(d+1,n)}{a_n^{(s)}}} \|f\|_{\Hs}.
\end{equation*}
\end{corollary}

\begin{remark}\label{rem:asymp}
If the number $m$ of quadrature points has a lower bound of order $n^d$, then $\|\mathcal{Q}_n f\|_{\Lt}$ is uniformly bounded by some constant. Recall from \eqref{equ:asl} that $a_{n}^{(s)}\asymp (1+n)^{-2s}$ and from \eqref{equ:numberZ} that $\quad Z(d+1,n) \sim \frac{2}{\Gamma(d+1)}n^{d}$ as $n\rightarrow\infty$. Thus if $m$ has a lower bound of order $n^{d+\frac{d^2}{2s}}$, then $\|\mathcal{Q}_n f-f\|_{\Lt}$ is uniformly bounded by some constant as $n\rightarrow\infty$. Moreover, if $m$ has a lower bound of order 
\begin{equation}\label{equ:order2}
(n+1)^{d+\varepsilon_1}n^{\frac{d^2}{2s}+\varepsilon_2}
\end{equation}
where $\varepsilon_1,\varepsilon_2>0$, then $\|\mathcal{Q}_n f-f\|_{\Lt}\rightarrow 0$ as $n\rightarrow\infty$
\end{remark}

If the QMC hyperinterpolation is regarded as a special case of the unfettered hyperinterpolation, then the expression \eqref{equ:order1} on $\eta$ requires $m$ to have a lower bound of order 
\begin{equation}\label{equ:order3}
(2n+1)^{d+\varepsilon_1}n^{\frac{2d^2-d}{2s}+\varepsilon_2}
\end{equation}
so that $\eta\rightarrow 0$ and hence $\|\mathcal{Q}_n f-f\|_{\Lt}\rightarrow 0$ as $n\rightarrow\infty$. For the same values of $\varepsilon_1$ and $\varepsilon_2$, the order \eqref{equ:order3} derived from regarding the QMC hyperinterpolation as a special case of the unfettered hyperinterpolation is unconditionally greater than the order \eqref{equ:order2} derived from Theorem \ref{QMCthm}, as $\frac{d^2}{2s}<\frac{d^2}{s}-\frac{d}{2s}$ holds for any $d\geq 1$. Moreover, as the term $E_n(f)$ in the estimate \eqref{equ:error} in Theorem \ref{thm} also has convergence rate of $n^{-s}$, what essentially varies the general estimate \eqref{equ:error} and the refined estimate \eqref{equ:errorqmc} is the other term in both estimates: the term $\sqrt{\eta^2+4\eta}\|\chi^*\|_{\Lt}$ in the estimate \eqref{equ:error} and the term $\frac{1}{m^{s/d}}\sqrt{\frac{Z(d+1,n)}{a_n^{(s)}}} \|f\|_{\Hs}$ in the refined estimate \eqref{equ:errorqmc}. For a fixed degree $n$, we have demonstrated in Section \ref{subsec:QMCinferior} that the convergence rate of the term in \eqref{equ:error} with respect to $m$ is $m^{-s/(2d)}$, and we can see the convergence rate of the term in \eqref{equ:errorqmc} is $m^{-s/d}$.

\begin{corollary}\label{cor:QMChyper}
With the aid of Remark \ref{rem:asymp}, we know that if $E_n(f)\lesssim{n^{-s}}$, then letting $m\gtrsim (n+1)^dn^{\frac{d^2}{2s}}n^{d}$ gives
\begin{equation*}
\|\mathcal{Q}_nf - f\|_{\Lt}\lesssim {n^{-s}}.
\end{equation*}
\end{corollary}

\begin{remark}
For the above results, we assume $f\in \Hs(\Sd)$ and $\{x_j\}_{j=1}^m$ is a QMC design for $\Hs(\Sd)$. Recall the concept of QMC strength. Suppose $f\in H^{s'}$ and $\{x_j\}_{j=1}^m$ is a QMC design with strength $s^*$, then $s$ in the above results should be $s = \min\{s',s^*\}$. 
\end{remark}

\section{Numerical experiments}\label{sec:numerical}

\subsection{Point sets and test functions}
Many different sequences of point sets on the sphere have been introduced in the literature. In the following experiments, we use points sets including
\begin{itemize}
       \item[$\circ$] Random scattered points generated by the following MATLAB commands:\\
       \texttt{rvals = 2*rand(m,1)-1;}\\
       \texttt{elevation = asin(rvals); }\texttt{\% calculate an elevation angle for each point}\\
       \texttt{azimuth = 2*pi*rand(m,1); }\texttt{\% create an azimuth angle for each point}\\
       \texttt{\% convert to Cartesian coordinates}\\
       \texttt{[x1,x2,x3] = sph2cart(azimuth,elevation,ones(m,1));}
       \item[$\circ$] Equal area points \cite{MR1306011} based on an algorithm given in \cite{MR2582801};
       \item[$\circ$] Fekete points which maximize the determinant for polynomial interpolation \cite{MR2065291};
       \item[$\circ$] Coulomb energy points, which minimize $\sum_{i,j=1}^m(1/\|x_i-x_j\|_2)$;
       \item[$\circ$] Spherical $t$-designs.
\end{itemize}
Random scattered points are directly generated in MATLAB, equal area points are generated based on the Recursive Zonal Equal Area (EQ) Sphere Partitioning Toolbox by Leopardi, Fekete points and Coulomb energy points are computed by Womersley in advance and are available on his website\footnote{Robert Womersley, \emph{Interpolation and Cubature on the Sphere}, \url{http://
www.maths.unsw.edu.au/~rsw/Sphere/}; accessed in August, 2022.}, and spherical $t$-designs are generated as the so-called well conditioned spherical $t$-designs in \cite{MR2763659}.

Moreover, we consider four kinds of test functions, including 
\begin{itemize}
       \item[$\circ$] A polynomial $f_1(x)=(x_1+x_2+x_3)^2\in\mathbb{P}_6(\Sd)$;
       \item[$\circ$] $f_2(x_1,x_2,x_3): = \lvert x_1+x_2+x_3\rvert+ \sin^2(1+\lvert x_1+x_2+x_3 \rvert)$, which is continuous but non-smooth;
       \item[$\circ$] The Franke function for the sphere \cite[p. 146]{MR946761}
       \begin{equation*}\begin{split}
       f_3(x_1,x_2,x_3) :=& 0.75\exp(-((9x_1-2)^2)/4-((9x_2-2)^2)/4-((9x_3-2)^2)/4)\\
                                        & +0.75\exp(-((9x_1+1)^2)/49-((9x_2+1))/10-((9x_3+1))/10)\\
                                        & +0.5\exp(-((9x_1-7)^2)/4-((9x_2-3)^2)/4-((9x_3-5)^2)/4)\\
                                        & -0.2\exp(-((9x_1-4)^2)-((9x_2-7)^2)-((9x_3-5)^2)),
       \end{split}\end{equation*}
       which is in $C^{\infty}(\Sd)$;
       \item[$\circ$] The sums of six compactly supported Wendland radial basis function \cite{MR3668040}
       \begin{equation*}
              f_{4,\sigma}:=\sum_{i=1}^6\phi_{\sigma}(z_i-x),\quad \sigma\geq 0,
       \end{equation*}
       where $z_1=[1,0,0]^{\text{T}}$, $z_2 = [-1,0,0]^{\text{T}}$, $z_3 = [0,1,0]^{\text{T}}$, $z_4=[0,-1,0]^{\text{T}}$, $z_5 = [0,0,1]^{\text{T}}$, and $z_6=[0,0,-1]^{\text{T}}$. The original Wendland functions 
       \begin{equation*}
              \tilde{\phi}_{\sigma}(r):=
              \begin{cases}
              (1-r)_+^2,& \sigma =0,\\
              (1-r)^4_+(4r+1), & \sigma = 1,\\
              (1-r)_+^6(35r^2+18r+3)/3,& \sigma=2,\\
              (1-r)_+^8(32r^3+25r^2+8r+1),&\sigma=3,\\
              (1-r)_+^{10}(429r^4+450r^3+210r^2+50r+5)/5,&\sigma=4,
              \end{cases}
       \end{equation*}
       are defined in \cite{wendland1995piecewise}, where $(r)_+:=\max\{r,0\}$ for $r\in\mathbb{R}$, and the normalized Wendland functions (test functions below) as defined in \cite{MR3822234} are
       \begin{equation*}
       \phi_{\sigma}(r) : = \tilde{\phi}_{\sigma}\left(\frac{r}{\delta_{\sigma}}\right),\quad \delta_{\sigma}:=\frac{3(\sigma+1)\Gamma(\sigma+1/2)}{2\Gamma(\sigma+1)},\quad \sigma \geq 0.
       \end{equation*}
       The normalized Wendland functions converge pointwise to a Gaussian as $\sigma\rightarrow\infty$, see \cite{chernih2014wendland}; moreover, $f_{4,\sigma}\in H^{\sigma+3/2}(\Sd)$, see \cite{MR2740542,MR1920637}.
\end{itemize}

\subsection{Unfettered hyperinterpolation and scattered data}
We start with a very interesting example of the unfettered hyperinterpolation with scattered data. As we have discussed in Theorem \ref{thm} and Corollary \ref{cor:unfettered}, the performance (i.e., the $\Lt$ error) of the unfettered hyperinterpolation is heavily dependent on the constant $\eta$, and what we need to do is to control this constant. In particular, if the degree $n$ and the number $m$ of quadrature points are fixed, Corollary \ref{cor:unfettered} suggests that $\eta$ has a lower bound of order $\sqrt{n^2\log{n}/m}$. It is immediate to see that $\eta$ is positively correlated to $n$ and negatively to $m$. Moreover, the term $\sqrt{\eta^2+4\eta}\|\chi^*\|_{\Lt}$ in the error bound \eqref{equ:error} has a lower bound of order 
$$\sqrt{\frac{n^2\log{n}}{m}+4\sqrt{\frac{n^2\log{n}}{m}}}.$$
That is, for a given $n$, the term $\sqrt{\eta^2+4\eta}\|\chi^*\|_{\Lt}$ has a lower bound of order $m^{-1/4}$.

We first solely investigate the term $\sqrt{\eta^2+4\eta}\|\chi^*\|_{\Lt}$ that arises as an artifact when the quadrature exactness assumption is discarded and leads to the divergence of the unfettered hyperinterpolation by examining the test function $f_1\in\mathbb{P}_6(\Sd)$. As $E_n(f_1)=0$ for all $n\geq 6$, we can focus on this term $\sqrt{\eta^2+4\eta}\|\chi^*\|_{\Lt}$ by letting $n\geq 6$. The $\Lt$ errors are depicted in Figure \ref{fig:unfettered1}: For each pair of $(n,m)$, we test ten times and report the average in terms of solid lines with markers; the maximal and minimal errors among these ten tests contribute to the upper and lower bounds of the filled region. We have at least three observations. Firstly, a larger degree $n$ of the unfettered hyperinterpolation, counterintuitively but rigorously asserted by our theory, leads to a larger value of $\sqrt{\eta^2+4\eta}\|\chi^*\|_{\Lt}$, because Corollary \ref{cor:unfettered} suggests that $\eta$ is negatively related to $n$. Secondly, as $n$ increases, the unfettered hyperinterpolation becomes more stable in the sense that the gap between the maximal and minimal errors among the ten tests for each pair of $(n,m)$ shrinks. This is also asserted by Corollary \ref{cor:unfettered} that the error bound \eqref{equ:error} is valid with probability exceeding $1-\bar{c}n^{-\gamma}$. Thirdly, as $m$ increases, the decaying rate of the unfettered hyperinterpolation with respect to $m$ for each $n$ coincides with the rate of $m^{-1/4}$. This observation is partially covered by our theory that the term $\sqrt{\eta^2+4\eta}\|\chi^*\|_{\Lt}$ has a lower bound of order $m^{-1/4}$, see discussions in the previous paragraph, and we conjecture that there may hold $\sqrt{\eta^2+4\eta}\|\chi^*\|_{\Lt}\asymp m^{-1/4}$.
\begin{figure}[htbp]
  \centering
  \includegraphics[width=10cm]{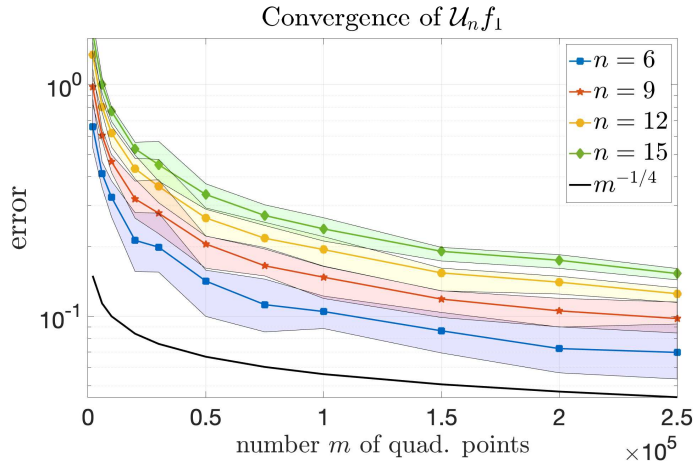}
  \caption{Convergence of the unfettered hyperinterpolation in the approximation of $f_1$.}\label{fig:unfettered1}
\end{figure}

After characterizing the behavior of the term $\sqrt{\eta^2+4\eta}\|\chi^*\|_{\Lt}$, we then consider the $\Lt$ error of the unfettered hyperinterpolation. If $E_n(f)$ is not zero, then error estimate \eqref{equ:error} is controlled by two terms, $E_n(f)$ and  $\sqrt{\eta^2+4\eta}\|\chi^*\|_{\Lt}$.  We repeat the above procedure for non-polynomial functions $f_2$ and $f_3$, and the $\Lt$ errors are displayed in Figure \ref{fig:unfettered2}, in which we only report the average errors. We see that when $m$ is relatively small, the term $\sqrt{\eta^2+4\eta}\|\chi^*\|_{\Lt}$ dominates the error bound, so a smaller $n$ leads to a smaller $\eta$ and hence a smaller error bound; when $m$ is relatively large, $\eta$ becomes tiny, and the term $E_n(f)$ dominates the error bound, so a larger $n$ leads to a smaller error bound.

\begin{figure}[htbp]
  \centering
  \includegraphics[width=\textwidth]{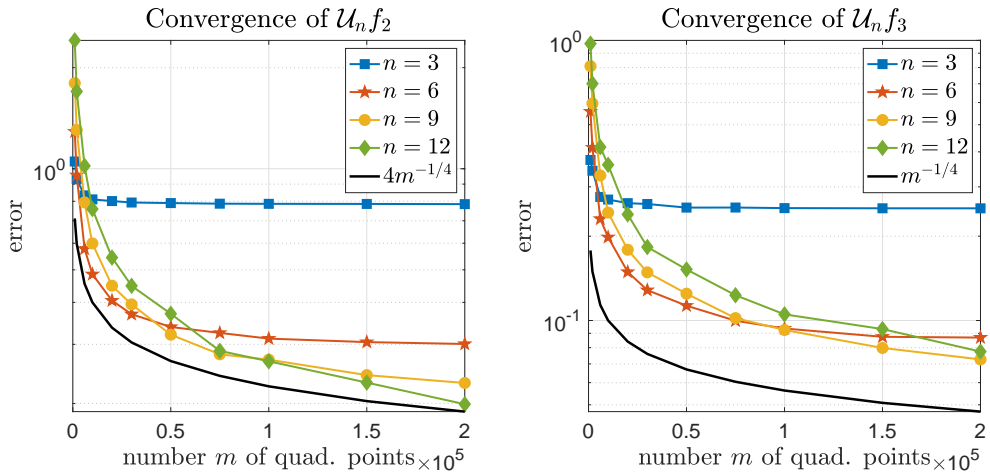}
  \caption{Convergence of the unfettered hyperinterpolation in the approximation of $f_2$ and $f_3$.}\label{fig:unfettered2}
\end{figure}

Thus, we may conclude a \emph{rule of thumb} for determining the degree $n$ of the unfettered hyperinterpolation in real-world applications: If the number of samples is limited, then choose a small $n$; on the other hand, if the samples are relatively sufficient, then choose a large $n$.

\subsection{QMC hyperinterpolation and QMC designs}

We then investigate the QMC hyperinterpolation, using equal area points, Coulomb energy points, Fekete points, and spherical $t$-designs. We first consider the approximation of $f_1\in\mathbb{P}_6$ by the QMC hyperinterpolation using equal area points, and we show that the refined error estimate \eqref{equ:errorqmc} in Theorem \ref{QMCthm} is indeed sharper than the estimate \eqref{equ:error} in Theorem \ref{thm}. A convergence result of quadrature rules using equal area points can be found in \cite[Section 6.1]{hesse2010numerical}. For any $n\geq 6$, we have 
\begin{equation}\label{equ:errorf1}
\|\mathcal{Q}_nf_1-f_1\|_{\Lt}\leq \frac{c''(s,d)}{m^{s/d}}\sqrt{\frac{Z(d+1,n)}{a_n^{(s)}}} \|f_1\|_{\Hs},
\end{equation}
in the light of Corollary \ref{cor:QMCpoly}. As the QMC strength $s^*$ of equal area points is conjectured in \cite{MR3246811} to be $2$, we may expect the decaying rate of $\|\mathcal{Q}_nf_1-f_1\|_{\Lt}$ with respect to $m$ to be $m^{-1}$ on the 2-sphere $\mathbb{S}^2$. However, from the general framework of the unfettered hyperinterpolation, we can only expect the decaying rate to be $m^{-1/2}$; see discussions in Section \ref{subsec:QMCinferior}. The $\Lt$ errors are depicted in Figure \ref{fig:QMC1}, which perfectly coincide with these deductions from our theory. We see that although the QMC hyperinterpolation can be regarded as a special case in the general framework of unfettered hyperinterpolation, the general estimate may not be sharp. Moreover, we find that a smaller $n$ leads to a smaller error, as suggested by the error bound \eqref{equ:errorf1}.

\begin{figure}[htbp]
  \centering
  \includegraphics[width=9cm]{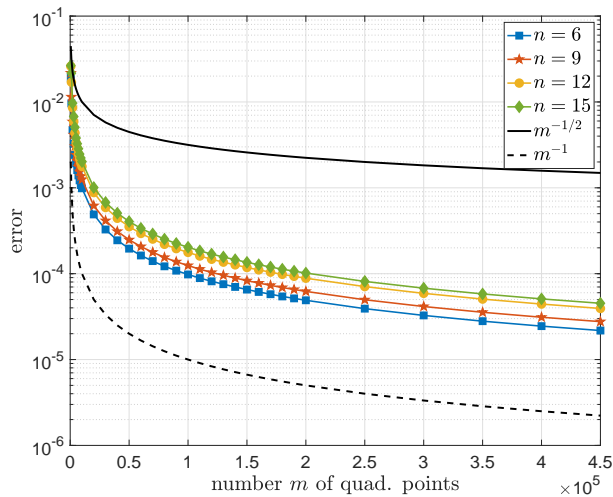}
  \caption{Convergence of the QMC hyperinterpolation in the approximation of $f_{1}$ using equal area points.}\label{fig:QMC1}
\end{figure}

We then consider the approximation of the normalized Wendland function $f_{4,2}$ by QMC hyperinterpolation, in which the term $n^{-s}\|f_{4,2}\|_{\Hs}$ cannot be ignored. Thus, the terms $n^{-s}$ and $m^{-s/2}$ jointly determine the convergence rate of $\|\mathcal{Q}_nf_{4,2}-f_{4,2}\|_{\Lt}$. It is conjectured in \cite{MR3246811} that the strength of Fekete points, equal area points, and Coulomb energy points is 1.5, 2, and 2, respectively. The $\Lt$ errors are depicted in Figure \ref{fig:QMC2}. Similarly to the unfettered hyperinterpolation using scattered data, we see that the term $\frac{1}{m^{s/d}}\sqrt{\frac{Z(d+1,n)}{a_n^{(s)}}}\|f\|_{\Hs}$ dominates the error bound when $m$ is relatively small, so a smaller $n$ leads to a smaller error; and the term $n^{-s}\|f\|_{\Hs}$ dominates the error bound when $m$ is relatively large. We observe that each error curve flattens as $m$ increases, and the curve of $n=6$ is higher than others when $m$ is large enough. Note that each curve corresponds to a fixed degree $n$. Thus the rule of thumb for determining the degree $n$ of the unfettered hyperinterpolation also applies to the QMC hyperinterpolation. The error curves of the QMC hyperinterpolation using spherical $t$-designs quickly flatten once the number $m$ of spherical $t$-designs renders the required quadrature exactness degrees. The convergence of the QMC hyperinterpolation using Fekete points is not monotonic. In light of Womersley's caveat on his website, the non-monotonic convergence is possibly caused by the fact that all computed Fekete points are only approximate local maximizers of the determinant for polynomial interpolation.

\begin{figure}[htbp]
  \centering
  \includegraphics[width=\textwidth]{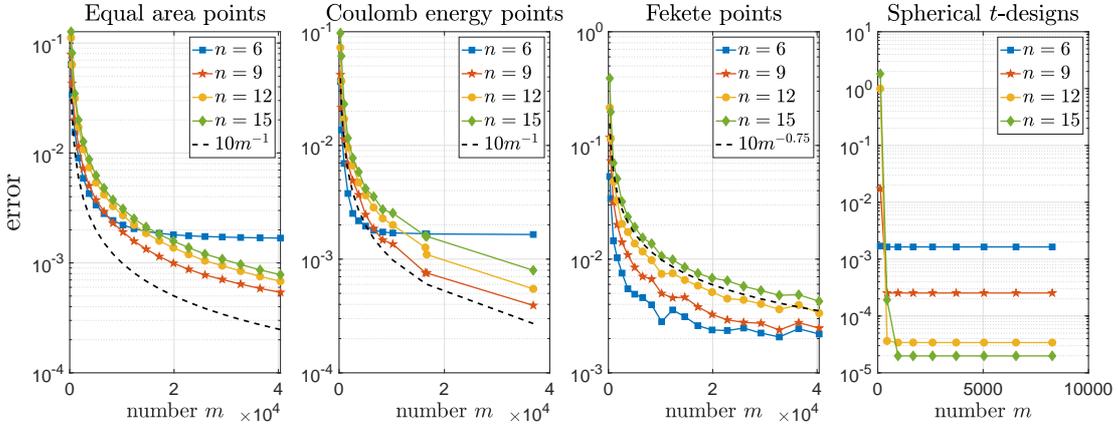}
  \caption{Convergence of the QMC hyperinterpolation in the approximation of $f_{4,2}$ using different kinds of point sets.}\label{fig:QMC2}
\end{figure}

We then study the performance of the QMC hyperinterpolation in the approximation of functions with different levels of smoothness. As we mentioned, the normalized Wendland function $f_{4,\sigma}$ belongs to $H^{\sigma+3/2}(\Sd)$. The $\Lt$ errors of the QMC hyperinterpolation of degree $n=5$ in the approximation of $f_{4,\sigma}$ with $\sigma = 0,1,\ldots,4$ are displayed in Figure \ref{fig:QMC3}, and the degree is intentionally set so small that error curves corresponding to different $\sigma$ can be distinguished. As we expect, the QMC hyperinterpolation is better in terms of $\Lt$ errors if the function to be approximated is smoother.

\begin{figure}[htbp]
  \centering
  \includegraphics[width=\textwidth]{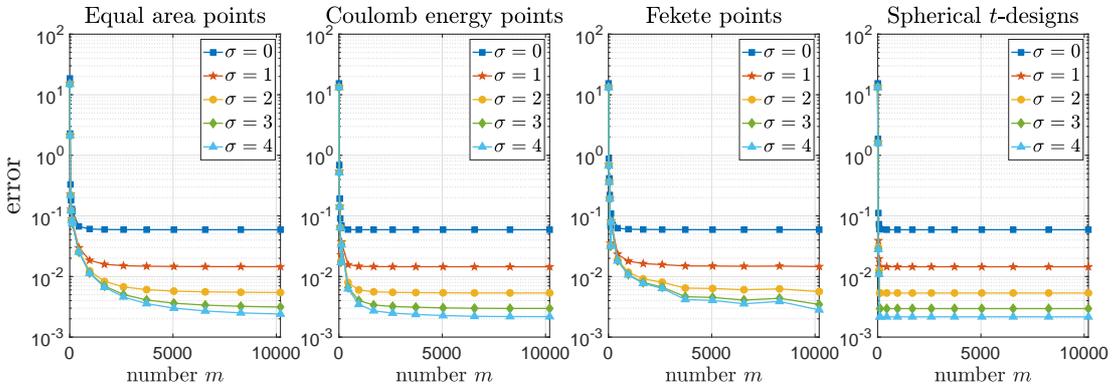}
  \caption{Convergence of the QMC hyperinterpolation in the approximation of $f_{4,\sigma}$ with $\sigma = 0,1,2,3,4$.}\label{fig:QMC3}
\end{figure}

Finally, we give a numerical example related to Remark \ref{rem:asymp} and Corollary \ref{cor:QMChyper} by considering the approximation of $f_{4,\sigma}$. As we mentioned in Section \ref{sec:designs}, to form a spherical $t$-design, $m$ should satisfy $m\asymp t^d$. Thus, to construct an original hyperinterpolant $\mathcal{L}_nf$ of degree $n$ on the 2-sphere $\mathbb{S}^2$ requires $m$ to be of order $n^2$, and we have $\|\mathcal{L}_nf-f\|_{\Lt}\rightarrow0$ as $n\rightarrow\infty$. According to Remark \ref{rem:asymp}, $m$ should have a lower bound of order $(n+1)^{d+\varepsilon_1}n^{\frac{d^2}{2s}+\varepsilon_2}$ for any $\varepsilon_1,\varepsilon_2>0$ to imply $\|\mathcal{Q}_nf-f\|_{\Lt}\rightarrow0$ as $n\rightarrow\infty$. The $\Lt$ errors with respect to the degree $n$ are depicted in Figure \ref{fig:QMC4}, and we let $m=(n+1)^2$ and $\lceil(n+1)^2n^{\frac{2}{\sigma+3/2}}\rceil$. The choice of $m=(n+1)^2$, which suffices to ensure the convergence of the original hyperinterpolation as $n\rightarrow\infty$, fails to imply the \emph{monotonic} convergence of the QMC hyperinterpolation. The choice of $m=\lceil(n+1)^2n^{\frac{2}{\sigma+3/2}}\rceil$, according to our theory, can ensure the convergence of $\mathcal{Q}_nf$ as $n\rightarrow \infty$, as shown in Figure \ref{fig:QMC4}. It may be strange to find that a larger $\sigma$ leads to a larger error level; this is due to the choice of $m=\lceil(n+1)^2n^{\frac{2}{\sigma+3/2}}\rceil$: a larger $\sigma$ implies a smaller $m$.

\begin{figure}[htbp]
  \centering
  \includegraphics[width=\textwidth]{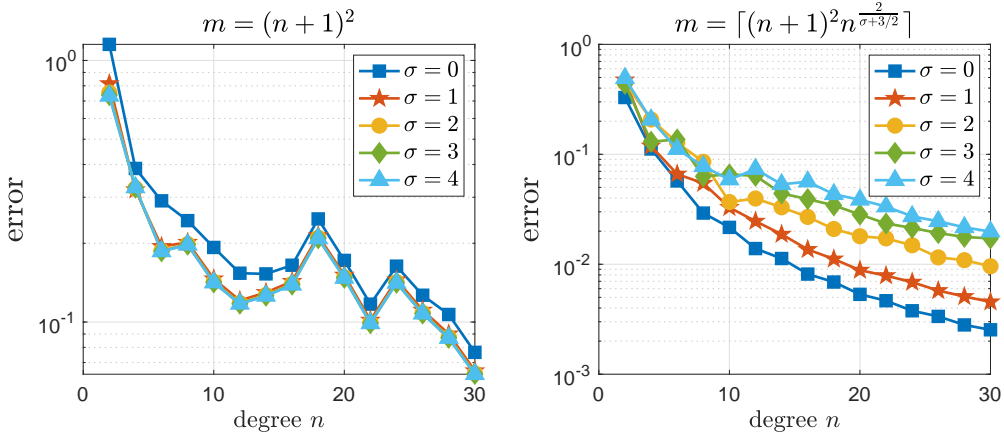}
  \caption{Performance of the QMC hyperinterpolation in the approximation of $f_{4,\sigma}$ with $m = (n+1)^2$ and $m=\lceil(n+1)^2n^{\frac{2}{\sigma+3/2}}\rceil$.}\label{fig:QMC4}
\end{figure}

By Corollary \ref{cor:QMChyper}, if we let $m\gtrsim (n+1)^2n^{\frac{2}{s}+2}$, then we can expect $\|\mathcal{Q}_nf-f\|_{\Lt}\lesssim n^{-s}$. This corollary is asserted by Figure \ref{fig:QMC5}, in which we investigate the approximation of $f_{4,2}$. We know that $f_{4,2}\in H^{2+3/2}(\Sd)$, thus we test on five choices of the number $m$, namely, $m = \beta\lceil(n+1)^2n^{2+\frac{2}{2+3/2}}\rceil$ with $\beta = 1,2,\ldots,5$. We see that the decaying rates of five choices all coincide with $m^{-(2+3/2)}$. This observation suggests $\|\mathcal{Q}_nf_{4,2}-f_{4,2}\|_{\Lt}\lesssim n^{-(2+3/2)}$, and more importantly, successfully verifies our theory on the QMC hyperinterpolation.

\begin{figure}[htbp]
  \centering
  \includegraphics[width=9cm]{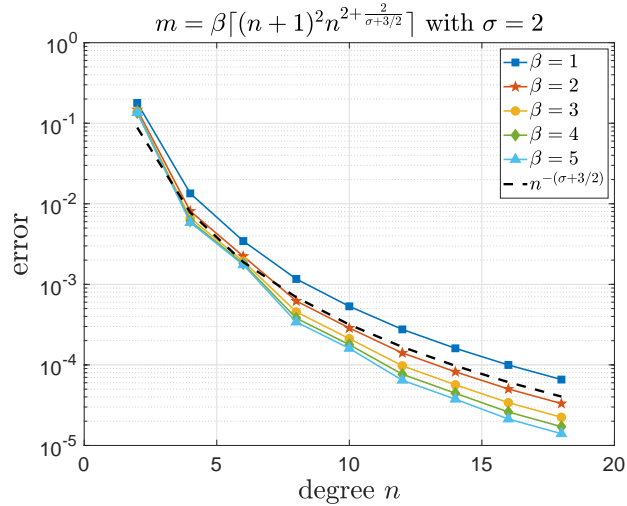}
  \caption{Convergence of the QMC hyperinterpolation in the approximation of $f_{4,2}$ with $m$ been a multiple $\beta$ of $\lceil(n+1)^2n^{2+\frac{2}{\sigma+3/2}}\rceil$ for $\beta = 1,2,\ldots,5$.}\label{fig:QMC5}
\end{figure}

\section{Concluding Remarks}

In this paper, we investigate the approximation scheme of hyperinterpolation on the sphere. The quadrature rules used in the construction of hyperinterpolation are not required to be exact for any polynomials but only to satisfy the Marcinkiewicz--Zygmund property, and we give the corresponding error estimate. Such an approximation scheme without the quadrature exactness assumption is referred to as the unfettered hyperinterpolation. If the quadrature rules use QMC designs, then the error estimate can be refined. To emphasize the particularity of QMC designs, we refer to the hyperinterpolation using QMC designs as quadrature points as the QMC hyperinterpolation. Note that the QMC hyperinterpolation can be regarded as a special case in the general framework of the unfettered hyperinterpolation. The general and refined estimates are split into two terms: a term representing the error estimate of the original hyperinterpolation of full quadrature exactness and another term introduced as compensation for the loss of exactness degrees. The newly introduced term may not converge to zero as the degree of hyperinterpolation tends to $\infty$, and we need to control it in practice. The numerical experiments show that the construction of hyperinterpolation using quadrature rules without exactness is feasible, and they verify the error estimates given in Sections \ref{sec:unfetteredtheory} and \ref{sec:QMCtheory}. The general framework of the unfettered hyperinterpolation on the sphere may be extended to the scheme of hyperinterpolation on other regions, such as a disk \cite{hansen2009norm}, a square \cite{caliari2007hyperinterpolation}, a cube \cite{caliari2008hyperinterpolation,wang2014norm}, a spherical triangle \cite{sommariva2021numerical}, and a spherical shell \cite{MR3554421,MR3739962}.

\section*{Acknowledgements}
We would like to thank Dr. Yoshihito Kazashi for his comment on our manuscript.




\bibliographystyle{siamplain}
\bibliography{myref}
\clearpage

\end{document}